\documentclass[12pt]{amsart}
\usepackage{amsmath}
\usepackage{amssymb}
\usepackage{latexsym}

\newcommand{\Cplx}{\mathbb {C}}     
\newcommand{\halmos}{\rule{5pt}{5pt}}

\numberwithin{equation}{section}

\newtheorem{prop}{\bf Proposition}[section]
\newtheorem{thm}[prop]{\bf Theorem}

\theoremstyle{definition}

\begin{document}
\title[On symmetry of $q$-Painlev\'e equations and associated linear equations]
{On symmetry of $q$-Painlev\'e equations and associated linear equations}
\author{Kouichi Takemura}
\address{Address of K.T., Department of Mathematics, Ochanomizu University, 2-1-1 Otsuka, Bunkyo-ku, Tokyo 112-8610, Japan}
\email{takemura.kouichi@ocha.ac.jp}
\begin{abstract}
We investigate the symmetry of the linear $q$-difference equations which are associated with some $q$-Painlev\'e equations.
We apply it for adjustment of the expression of the time evolution on the $q$-Painlev\'e equations in terms of the Weyl group symmetry.
\end{abstract}
\maketitle

\section{Introduction}
The Painlev\'e-type equations are non-linear generalizations of the hypergeometric-type equations.
Several members of the discrete analogue of the Painlev\'e equation had been discovered individually in the 1990's, and a comprehensive list of the second order discrete Painlev\'e equations was provided by Sakai \cite{Sak}.
Each member of the $q$-difference Painlev\'e equation was essentially labelled by the affine root system from the symmetry.
Sakai \cite{Sak} realized the time evolution of the discrete Painlev\'e equations by the symmetry of the Weyl group, which originates from the Cremona action on a family of surfaces which are related with the space of initial conditions.

The $q$-Painlev\'e equation of type $D^{(1)}_5 $, which is also denoted by $q$-$P(D^{(1)}_5)$, is given by
\begin{equation}
f\overline{f} = \nu_3\nu_4\, \frac{ ( g - \nu_5 /\kappa_2 ) ( g - \nu_6/\kappa_2 )}{ ( g - 1/\nu_1 ) ( g - 1/\nu_2) }, \quad
g\underline{g} = \frac{1}{\nu_1\nu_2}\frac{( f - \kappa_1 /\nu_7 ) ( f - \kappa_1/\nu_8 )}{(f - \nu_3)(f - \nu_4)}. \label{eq:qPD15}
\end{equation}
Here, $\overline{f} $ (resp. $\underline{f}$) is the consequence of the time evolution $t \mapsto qt$ (resp. the inverse time evolution $t \mapsto t/q$) to $f$.
The time evolution for the parameters is given by $\overline{\kappa }_1= \kappa_{1} /q$, $\overline{\kappa }_2= q \kappa_{2}$ and $ \overline{\nu }_i= \nu _{i}$ $(i=1,2,\dots ,8)$.
Note that the equation $q$-$P(D^{(1)}_5)$ was originally introduced by Jimbo and Sakai \cite{JS} as a $q$-analogue of the sixth Painlev\'e equation, and it was obtained by considering a $q$-analogue of monodromy preserving deformation.
Yamada studied the Lax pairs of the discrete Painlev\'e equations extensively \cite{Ye,Y}, and a list of the Lax pairs was provided in the survey \cite{KNY} by Kajiwara, Noumi and Yamada.
A Lax pair of the $q$-Painlev\'e equation $q$-$P(D^{(1)}_5)$ was given as
\begin{align}
L_{1} &= \Bigl\{\frac{z(g\nu_{1}-1)(g\nu_{2}-1)}{qg}-\frac{\nu_{1}\nu_{2}\nu_{3}\nu_{4}(g- \nu_{5}/\kappa_{2})(g- \nu_{6}/\kappa_{2})}{fg}\Bigr\} \label{eq:D5L1} \\
&+\frac{\nu_{1}\nu_{2}(z- q \nu_{3})(z - q \nu_{4})}{q(qf-z)}(g-T^{-1}_{z})+\frac{(z- \kappa_{1} /\nu_{7})(z - \kappa_{1}/\nu_{8})}{q(f-z)}\Bigl(\frac{1}{g} - T_{z}\Bigr) , \nonumber \\
L_{2} &= \Bigl(1-\frac{f}{z} \Bigr)T+T_{z}-\frac{1}{g} . \nonumber 
\end{align}
Here $T_{z}$ represents the transformation $z\mapsto qz$ and $T$ represents the time evolution such as $T(f)=\overline{f}$.
The equation $q$-$P(D^{(1)}_5)$ is obtained as the compatibility condition for the Lax operators $L_1$ and $L_2$.
In this paper we investigate the symmetry of the linear $q$-difference equation $L_1 y(z)=0$, where $L_1$ (see Eq.~(\ref{eq:D5L1})) is the Lax operator of the $q$-Painlev\'e equation of type $D^{(1)}_5$.
Typical symmetries of the equation $L_1 y(z)=0$ are given by the gauge transformations $y(z) \mapsto p(z) y(z)$ for some specific functions $p(z)$.
We also investigate the symmetry of the associated linear $q$-difference equation with other $q$-Painlev\'e equations denoted by $q$-$P(E^{(1)}_6)$ and $q$-$P(E^{(1)}_7)$.
The symmetry of the linear $q$-difference equation may induce the transformation of the parameters and some of them are described in terms of the affine Weyl group of the $q$-Painlev\'e equation.
The gauge transformation $y(z) \mapsto z^{\lambda } y(z)$ for $\lambda \in \Cplx $ and the dilation $z \mapsto a z$ for $a \in \Cplx \setminus \{ 0 \}$ also induce the transformation of the parameters.

It is known that the symmetry of each member of the $q$-Painlev\'e equation was described by using the affine Weyl group, and a systematic description was provided in \cite{KNY}. 
However, it seems that some adjustment would be necessary to fit the time evolution from the Lax pair with the Weyl group symmetry described in \cite{KNY}.
In this paper we describe the expression of the time evolution of $q$-Painlev\'e equations $q$-$P(D^{(1)}_5)$, $q$-$P(E^{(1)}_6)$ and $q$-$P(E^{(1)}_7)$ in terms of the Weyl group symmetry and the symmetry coming from the gauge transformation and the dilation.

This paper is organized as follows.
In section \ref{sec:D15}, we review the Weyl group action for the $q$-Painlev\'e equations $q$-$P(D^{(1)}_5)$ by following \cite{KNY}, and investigate the symmetry for the linear $q$-difference equation $L_1 y(z)=0 $.
We realize the time evolution of $q$-$P(D^{(1)}_5)$ in terms of the actions of the generators of the Weyl group and the symmetry coming from the linear $q$-difference equation.
In section \ref{sec:E16}, we introduce the $q$-Painlev\'e equations $q$-$P(E^{(1)}_6)$.
We also review the Lax pair and its Weyl group action of $q$-$P(E^{(1)}_6)$ by following \cite{KNY}, and investigate the symmetry for the corresponding linear $q$-difference equation.
We also investigate the time evolution of $q$-$P(E^{(1)}_6)$, which was inspired by the study of the difference Painlev\'e equation by Tsuda \cite{Td}.
In section \ref{sec:E17}, we give similar results for the $q$-Painlev\'e equations $q$-$P(E^{(1)}_7)$.

\section{The $q$-Painlev\'e equation of type $D^{(1)}_5 $} \label{sec:D15}
We review the Weyl group symmetry of the $q$-Painlev\'e equation $q$-$P(D^{(1)}_5 )$.
For this purpose, we recall the action of the operators $s_0, \dots ,s_5$, $\pi _1$ and $\pi _2$ for the parameters $(\kappa_{1}, \kappa_{2}, \nu_{1}, \dots , \nu_{8}) $ and $(f, g) $ introduced in \cite{KNY}.
\begin{align}
s_0 : & \: \nu_7 \leftrightarrow \nu_8 , \quad s_1 : \: \nu_3 \leftrightarrow \nu_4 , \quad  s_4 : \: \nu_1 \leftrightarrow \nu_2 , \quad  s_5 : \: \nu_5 \leftrightarrow \nu_6 , \label{eq:AffineWeylActionD5} \\
s_2 : & \: \nu_3 \rightarrow\frac{\kappa_1}{\nu_7}, \; 
 \nu_7 \rightarrow \frac{\kappa_1}{\nu_3}, \;
 \kappa_2 \rightarrow \frac{\kappa_1\kappa_2}{\nu_3\nu_7}, \;
 g \rightarrow g\, \frac{f - \nu_3}{f - \kappa_1 /\nu_7} , \nonumber \\
s_3 : & \: \nu_1 \rightarrow\frac{\kappa_2}{\nu_5}, \;
 \nu_5 \rightarrow \frac{\kappa_2}{\nu_1}, \;
 \kappa_1 \rightarrow \frac{\kappa_1\kappa_2}{\nu_1\nu_5}, \;
 f \rightarrow f\, \frac{g - 1/\nu_1}{g - \nu_5/\kappa_2} , \nonumber \\
  \pi _1 : & \: q \rightarrow 1/q, \: \nu_1 \rightarrow 1/\nu_1, \:  \nu_2 \rightarrow 1/\nu_2,  \: \nu_3 \rightarrow 1/\nu_7,  \: \nu_4 \rightarrow 1/\nu_8, \: \nu_5 \rightarrow 1/\nu_5, \: \nu_6 \rightarrow 1/\nu_6, \nonumber \\
& \: \nu_7 \rightarrow 1/\nu_3, \: \nu_8 \rightarrow 1/\nu_4, \: \kappa_1 \rightarrow 1/\kappa_1 , \: \kappa_2 \rightarrow 1/\kappa_2 , \: f \rightarrow f/\kappa_1 , \: g \rightarrow 1/g , \nonumber \\
  \pi _2 : & \: q \rightarrow 1/q, \: \nu_1 \rightarrow 1/\nu_7, \:  \nu_2 \rightarrow 1/\nu_8,  \: \nu_3 \rightarrow 1/\nu_5,  \: \nu_4 \rightarrow 1/\nu_6, \: \nu_5 \rightarrow 1/\nu_3, \: \nu_6 \rightarrow 1/\nu_4, \nonumber \\
& \: \nu_7 \rightarrow 1/\nu_1, \: \nu_8 \rightarrow 1/\nu_2, \: \kappa_1 \rightarrow 1/\kappa_2 , \: \kappa_2 \rightarrow 1/\kappa_1 , \: f \rightarrow 1/(\kappa_2 g) , \: g \rightarrow \kappa_1/f . \nonumber 
\end{align}
The omitted variables are invariant by the action, e.g.~$s_2 (f) =f$.
Then we can confirm that these operations satisfy the relations of the extended Weyl group $\widetilde{W}(D^{(1)}_{5})$ whose Dynkin diagram is as follows.\\
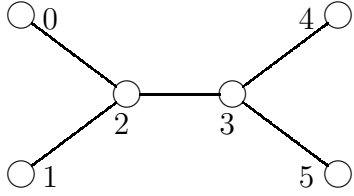
\begin{picture}(0,100)(0,0)
\put(70,80){\circle{10}}
\put(70,20){\circle{10}}
\put(190,80){\circle{10}}
\put(190,20){\circle{10}}
\put(110,50){\circle{10}}
\put(150,50){\circle{10}}
\qbezier(74,77)(90,65)(106,53)
\qbezier(74,23)(90,35)(106,47)
\qbezier(186,77)(170,65)(154,53)
\qbezier(186,23)(170,35)(154,47)
\qbezier(115,50)(130,50)(145,50)
\put(78,75){$0$}
\put(78,15){$1$}
\put(175,75){$4$}
\put(175,15){$5$}
\put(105,35){$2$}
\put(145,35){$3$}
\end{picture}
\\
Namely we have $ s_{i}^{2} = \mathrm{id}$, $(i = 0, \dots , 5) $, $\pi _1^2 = \pi_2^2 = \mathrm{id}$, $(\pi _1 \pi _2)^4 = \mathrm{id} $, $\pi _1 s_0 = s_1 \pi _1$, $\pi _1 s_j = s_j \pi _1$, $(j=2,3,4,5)$, $\pi _2 s_0 = s_4 \pi _2$, $\pi _2 s_1 = s_5 \pi _2$, $\pi _2 s_2 = s_3 \pi _2$ and
\begin{align}
  s_{i}s_{j}s_i = s_j s_i s_j , \quad & \{ i,j \} = \{ 0,2\} , \{ 1,2 \}, \{ 2,3 \} , \{ 3,4 \} , \{ 3,5 \}, \\
 s_{i}s_{j} = s_{j}s_{i} , \quad & \mbox{ otherwise. }  \nonumber
\end{align}
Recall that the operator $L_1$ is defined in Eq.~(\ref{eq:D5L1}) as one of the Lax pair, and the linear $q$-differential equation $L_1 y(z)= 0 $ is written as 
\begin{align}
& \Bigl\{ \frac{z(g\nu_{1}-1)(g\nu_{2}-1)}{qg}-\frac{\nu_{1}\nu_{2}\nu_{3}\nu_{4}(g- \nu_{5}/\kappa_{2})(g- \nu_{6}/\kappa_{2})}{fg} \Bigr\} y(z) \label{eq:D5L1yz} \\
& +\frac{\nu_{1}\nu_{2}(z - q \nu_{3})(z - q \nu_{4})}{q(qf- z)}(g y(z) -y(z/q) ) \nonumber \\
& +\frac{(z- \kappa_{1}/\nu_{7})(z- \kappa_{1}/\nu_{8})}{q(f-z)}\Bigl( \frac{1}{g} y(z)- y(qz) \Bigr) =0 . \nonumber 
\end{align}
Note that Eq.~(\ref{eq:D5L1yz}) was studied in \cite{STT1} from the aspect of the initial-value space of $q$-$P(D^{(1)}_5 )$, and it was observed that the $q$-Heun equation appears.

In this paper, we investigate the action of the symmetry in Eq.~(\ref{eq:AffineWeylActionD5}) on Eq.~(\ref{eq:D5L1yz}).
It follows immediately that Eq.~(\ref{eq:D5L1yz}) is invariant under the actions of $s_0 $, $s_1$, $s_4 $ and $s_5$.
It was discussed in \cite{STT2} that the action of $s_3$ is related to the $q$-middle convolution.

Next, we investigate the action of $s_2$ on Eq.~(\ref{eq:D5L1yz}).
It follows from Eq.~(\ref{eq:AffineWeylActionD5}) that $\nu _3  $ and $\kappa_{1}/\nu_{7} $ are exchanged by the action of $s_2$. 
Thus, we set 
\begin{equation}
 y(z)= \tilde{y}(z) \frac{(q\nu _3 /z;q)_{\infty }}{( q \kappa _1 / (\nu _7 z) ;q)_{\infty }}
= \tilde{y}(z) \prod _{j=0}^{\infty } \frac{1- q^{j +1} \nu _3 /z }{1-q^{j+1} \kappa _1 / (\nu _7 z) }
. \label{eq:D5L1s2}
\end{equation}
Then we have
\begin{align}
& y(q z)= \tilde{y}(q z) \frac{(1- \nu _3 /z)}{(1- \kappa _1 / (\nu _7 z))} \frac{(q\nu _3 /z;q)_{\infty }}{( q \kappa _1 / (\nu _7 z) ;q)_{\infty }}, \\
& y(z/q)= \tilde{y}(z/q) \frac{(1- q \kappa _1 / (\nu _7 z))}{(1- q\nu _3 /z)} \frac{(q\nu _3 /z;q)_{\infty }}{( q \kappa _1 / (\nu _7 z) ;q)_{\infty }} . \nonumber
\end{align}
Therefore, if $y(z)$ is a solution of Eq.~(\ref{eq:D5L1yz}), then $\tilde{y}(z)$ satisfies
\begin{align}
& \Bigl\{ \frac{z(g\nu_{1}-1)(g\nu_{2}-1)}{qg}-\frac{\nu_{1}\nu_{2}\nu_{3}\nu_{4}(g- \nu_{5}/\kappa_{2})(g- \nu_{6}/\kappa_{2})}{fg} \Bigr\} \tilde{y}(z) \label{eq:D5L1yz1} \\
& +\frac{\nu_{1}\nu_{2}(z - q \nu_{3})(z - q \nu_{4})}{q(qf- z)}g \tilde{y}(z) -\frac{\nu_{1}\nu_{2}(z - q \kappa _1 / \nu _7 )(z - q \nu_{4})}{q(qf- z)} \tilde{y}(z/q)  \nonumber \\
& +\frac{(z- \kappa_{1}/\nu_{7})(z- \kappa_{1}/\nu_{8})}{q(f-z)} \frac{1}{g} \tilde{y}(z) - \frac{(z- \nu_{3})(z- \kappa_{1}/\nu_{8})}{q(f-z)} \tilde{y}(qz) =0 . \nonumber
\end{align}
It follows from a straightforward calculation that Eq.~(\ref{eq:D5L1yz1}) is written as 
\begin{align}
& \left\{\frac{z(\tilde{g}\nu_{1}-1)(\tilde{g}\nu_{2}-1)}{q\tilde{g}}-\frac{\nu_{1}\nu_{2}\tilde{\nu }_{3}\nu_{4}(\tilde{g}- \nu_{5}/\tilde{\kappa}_{2})(\tilde{g}-\nu_{6}/\tilde{\kappa}_{2})}{f\tilde{g}}\right\} \tilde{y}(z) \\
& +\frac{\nu_{1}\nu_{2}(z - q \tilde{\nu }_{3})(z - q \nu_{4})}{q(qf- z)}(\tilde{g} \tilde{y}(z) -\tilde{y}(z/q) ) \nonumber \\
& +\frac{(z- \kappa_{1}/\tilde{\nu }_{7})(z- \kappa_{1}/\nu_{8})}{q(f-z)}\Bigl( \frac{1}{\tilde{g}} \tilde{y}(z)- \tilde{y}(qz) \Bigr) =0 \nonumber 
\end{align}
by setting $\tilde{\nu }_3 = \kappa_1 /\nu_7$, $\tilde{\nu }_7 = \kappa_1 /\nu_3 $, $\tilde{\kappa}_2 = \kappa_1 \kappa_2 /(\nu_3 \nu_7)$ and $\tilde{g} = g (f - \nu_3 )/(f - \kappa_1 /\nu_7) $.
Therefore, the gauge transformation defined by Eq.~(\ref{eq:D5L1s2}) induces the symmetry $s_2$ of the $q$-Painlev\'e equation of type $D^{(1)}_{5}$.
The gauge transformation defined by
\begin{equation}
 y(z)= \tilde{y}(z) \frac{(q\nu _3 /z;q)_{\infty }(q\nu _4 /z;q)_{\infty }}{( q \kappa _1 / (\nu _7 z) ;q)_{\infty }( q \kappa _1 / (\nu _8 z) ;q)_{\infty }}
\end{equation}
induces the symmetry $s_2 s_1 s_0 s_2 $ of the $q$-Painlev\'e equation of type $D^{(1)}_{5}$.

We examine another gauge-transformation.
Let $y(z)$ be a solution of Eq.~(\ref{eq:D5L1yz}) and set $y(z)= z^d \tilde{y} (z)$.
Then the function $\tilde{y} (z) $ satisfies
\begin{align}
& \Bigl\{ \frac{z(q^{d} g q^{-d }\nu_{1} -1)(q^d g q^{-d } \nu_{2} -1)}{q q^d g} \\
& \qquad -\frac{ q^{-d }\nu_{1} q^{-d }\nu_{2} \nu_{3}\nu_{4}(q^d g- q^d \nu_{5}/\kappa_{2}) (q^d g- q^d \nu_{6}/\kappa_{2})}{f q^d g} \Bigr\} \tilde{y} (z) \nonumber \\
& +\frac{q ^{-d} \nu_{1} q^{-d } \nu_{2}(z - q \nu_{3})(z - q \nu_{4})}{q(qf- z)}(q^d g \tilde{y} (z) -\tilde{y} (z/q) ) \nonumber \\
& +\frac{(z- \kappa_{1}/\nu_{7})(z- \kappa_{1}/\nu_{8})}{q(f-z)}\Bigl( \frac{1}{q^d g} \tilde{y} (z)- \tilde{y} (qz) \Bigr) =0 . \nonumber 
\end{align}
Hence, the correspondence of the parameters is written as 
$(\nu_1 , \nu_2 , \nu_5/\kappa_2 ,\nu_6/\kappa_2 ,g) \mapsto (q^{-d} \nu_1 , q^{-d} \nu_2 , q^d \nu_5/\kappa_2 ,q^d \nu_5/\kappa_2 ,q^d g) $, while the parameters $\nu_3 , \nu_4 , \nu_7/\kappa_1 ,\nu_8/\kappa_1 ,f $ are unchanged.
We denote by $G[s]$ the transformation of the parameters
\begin{align}
&  (\nu_1 , \nu_2 , \nu_3 , \nu_4 ,\nu_5/\kappa_2 ,\nu_6/\kappa_2 , \nu_7/\kappa_1 ,\nu_8/\kappa_1 ,f ,g) \label{eq:Gs} \\
& \mapsto (\nu_1 /s ,  \nu_2 /s , \nu_3 , \nu_4, s \nu_5/\kappa_2 ,s \nu_6/\kappa_2 , \nu_7/\kappa_1 ,\nu_8/\kappa_1 ,f ,s g) . \nonumber
\end{align}

We investigate the dilation $z \mapsto cz$ for the non-zero constant $c$.
Set $z =u/c$ and $y(z)=\tilde{y}(u) $.
Let $y(z)$ be a solution of Eq.~(\ref{eq:D5L1yz}).
Then the function $\tilde{y}(u)  $ satisfies
\begin{align}
& \Big\{ u \frac{(\nu _1 g-1)(\nu _2 g-1)}{q g} - \nu _1 \nu _2 c \nu_3 c \nu_4 \frac{ ( g - \nu_5/\kappa_2 ) ( g - \nu_6 /\kappa_2 ) }{c f g } \Big\} \tilde{y}(u) \\
& + \frac{ \nu _1 \nu _2 (u - q c \nu_3) (u - q c \nu_4) }{q(q c f -u )} \Big( g \tilde{y}(u)  - \tilde{y}(u/q) \Big) \nonumber \\
& + \frac{ (u - c \kappa_1/\nu_7)  (u - c \kappa_1 /\nu_8 ) }{q(c f -u )} \Big( \frac{1}{g}\tilde{y}(u) - \tilde{y}(qu) \Big) =0 . \nonumber 
\end{align}
Hence, the correspondence for the parameters is written as
\begin{align}
&  (\nu_1 , \nu_2 , \nu_3 , \nu_4 ,\nu_5/\kappa_2 ,\nu_6/\kappa_2 , \nu_7/\kappa_1 ,\nu_8/\kappa_1 ,f ,g) \label{eq:Dc}\\
& \mapsto (\nu_1 , \nu_2 , c \nu_3 , c \nu_4 ,\nu_5/\kappa_2 ,\nu_6/\kappa_2 ,  \nu_7/\kappa_1/c ,\nu_8/\kappa_1/c ,c f,g)  , \nonumber
\end{align}
which we denote by $D[c] $.

We investigate the symmetry related to $z \leftrightarrow 1/z$.
More precisely we take into account a more general situation $z =c/u$ and $y(z)= z^d \tilde{y}(u) = z^d \tilde{y}(c/z) $, where $y(z)$ is a solution of Eq.~(\ref{eq:D5L1yz}).
Then we have 
\begin{align}
& \Bigl\{ \frac{c(g\nu_{1}-1)(g\nu_{2}-1)}{u qg}-\frac{\nu_{1}\nu_{2}\nu_{3}\nu_{4}(g- \nu_{5}/\kappa_{2})(g- \nu_{6}/\kappa_{2})}{fg} \Bigr\}\tilde{y}(u) \\
& +\frac{\nu_{1}\nu_{2}(c/u - q \nu_{3})(c/u - q \nu_{4})}{q(qf- c/u)}(g \tilde{y}(u) - q^{-d} \tilde{y}(qu) ) \nonumber \\
& +\frac{(c/u- \kappa_{1}/\nu_{7})(c/u- \kappa_{1}/\nu_{8})}{q(f-c/u)}\Bigl( \frac{1}{g} \tilde{y}(u) - q^d \tilde{y}(u/q) \Bigr) =0 . \nonumber 
\end{align}
Therefore,
\begin{align}
& \Bigl\{ \frac{ u((1/g)\nu_{5}/\kappa_{2} -1)((1/g)\nu_{6}/\kappa_{2} -1)}{q /(g q^d)} -\frac{ c \nu_5 \nu_6 \nu_7 \nu_8 (1/g - \nu_{1})(1/g - \nu_{2})}{ q \kappa _1^2 \kappa _2^2 /(fg q^d) }  \Bigr\}\tilde{y}(u) \\
& +\frac{q^{2d} \nu _5 \nu _6 (u -c \nu_7 / \kappa_{1})(u -c \nu_8 / \kappa_{1})}{\kappa _2 ^2 q (c /f - u )}\Bigl( \frac{1}{g q^d} \tilde{y}(u) - \tilde{y}(u/q) \Bigr) \nonumber \\
& +\frac{ (u -c/(q \nu _3) )(u -c/(q \nu_{4} ) )}{q(c /(q f)  - u )}( q^d g \tilde{y}(u) - \tilde{y}(qu) ) =0 . \nonumber 
\end{align}
In the case $q^d=\kappa _2$ and $c=q \kappa _1 $, the equation is written as 
\begin{align}
& \Bigl\{ \frac{ u( \nu_{5}/( g\kappa_{2}) -1)(\nu_{6}/(g\kappa_{2}) -1)}{q /(g \kappa _2)} \\
& \qquad -\frac{ \nu_5 \nu_6 \nu_7 \nu_8 (1/(g \kappa _2) - \nu_{1}/\kappa _2 )(1/(g \kappa _2) - \nu_{2}/\kappa _2)}{ \kappa _1  /(fg \kappa _2) }  \Bigr\}\tilde{y}(u) \nonumber \\
& +\frac{\nu _5 \nu _6 (u -q \nu_7)(u -q \nu_8)}{q (q \kappa _1 /f - u )}\Bigl( \frac{1}{g \kappa _2} \tilde{y}(u) - \tilde{y}(u/q) \Bigr) \nonumber \\
& +\frac{ (u -\kappa _1/ \nu _3 )(u -\kappa _1/\nu_{4}  )}{q(\kappa _1 / f  - u )}( g \kappa _2 \tilde{y}(u) - \tilde{y}(qu) ) =0 , \nonumber 
\end{align}
and the correspondence of the parameters is described by the action of $\pi _2 \pi _1 \pi _2 \pi _1 (= \pi _1 \pi _2 \pi _1 \pi _2 )$,
\begin{align}
& (\nu_1 , \nu_2 , \nu_3 , \nu_4 ,\nu_5 ,\nu_6 , \nu_7 ,\nu_8 ;\kappa_1 ,\kappa_2 ; f ,g) \\
& \mapsto  (\nu_5 ,\nu_6 , \nu_7 ,\nu_8 , \nu_1 , \nu_2 , \nu_3 , \nu_4 ;\kappa_1 ,\kappa_2 ; \kappa _1/ f , 1/(\kappa _2 g) ).  \nonumber
\end{align}

On $q$-$P(D^{(1)}_5)$, a candidate of the time evolution by the symmetry of $\widehat{W} (D^{(1)}_5) $ is $(\pi _2 \pi _1 s_2 s_1 s_0 s_2 )^2$, which was essentially written in Sakai \cite{Sak}.
By a direct calculation, we have
\begin{align}
& (\pi _2 \pi _1 s_2 s_1 s_0 s_2 )^2 (f) =\pi _2 \pi _1 s_2 s_1 s_0 s_2 (1/g ) = \frac{q \nu_3 \nu_4 \nu_7 \nu_8 }{ f \kappa_1 } \frac{(g - \nu _5/\kappa _2 )(g  - \nu _6/\kappa _2 )}{( g -1/\nu_1 )( g -1 /\nu_2 )}, 
\end{align} 
which does not coincide with 
\begin{align}
& \overline{f} = \frac{ \nu_3 \nu_4  }{ f  } \frac{(g - \nu _5 /\kappa _2 )(g  - \nu _6 /\kappa _2)}{( g -1/\nu_1 )( g -1 /\nu_2 )} .
\end{align}
Moreover we have $(\pi _2 \pi _1 s_2 s_1 s_0 s_2 )^2 (\nu _3) = q \nu _3 \nu_7 \nu_8 / \kappa_1 (\neq \nu _3)$.
We can adjust the expressions as the following theorem.
\begin{thm} \label{thm:D5TW} (i)
Let $\Xi $ be the transformation of the parameters defined by
\begin{align}
& (\nu_1 , \nu_2 , \nu_3 , \nu_4 ,\nu_5 ,\nu_6 , \nu_7 ,\nu_8 ;\kappa_1 ,\kappa_2 ; f ,g) \label{eq:D5Xi} \\
&  \mapsto \Bigl( \frac{\nu _1 \nu_ 5 \nu_6 }{ \kappa_2 } , \frac{\nu _2 \nu_ 5 \nu_6}{\kappa_2} , \frac{\kappa_1 }{q \nu_4 } , \frac{\kappa_1 }{q \nu_3 } ,\frac{\nu _5 \nu_ 1 \nu_2 }{\kappa_2 } ,\frac{\nu _6 \nu_ 1 \nu_2 }{\kappa_2 } , \frac{\kappa_1 }{q \nu_8 } , \frac{\kappa_1 }{q \nu_7 }  ; \nonumber\\
& \qquad \qquad \qquad \qquad \frac{\kappa_1^3}{q^2 \nu_3 \nu_4 \nu_7 \nu_8 } ,\frac{\nu_1 \nu_2 \nu_5 \nu_6 }{\kappa _2 } ; \frac{f \kappa_1}{q \nu_3 \nu_4 } , \frac{g \kappa_2 }{\nu_ 5 \nu_6 } \Bigr) .  \nonumber
\end{align}
Let $T$ be the transformation of the parameters defined by 
\begin{equation}
T= \Xi \cdot (\pi _2 \pi _1 s_2 s_1 s_0 s_2 )^2 .
\end{equation}
Then 
\begin{align}
& T(f)= \frac{ \nu_3 \nu_4  }{ f  } \frac{(g - \nu _5 /\kappa _2 )(g  - \nu _6 /\kappa _2)}{( g -1/\nu_1 )( g -1 /\nu_2 )}, \label{eq:D5TfTg} \\
&  T(g)=  \frac{1}{g \nu_1 \nu_2 }\frac{( T(f) - T(\kappa_1 /\nu_7)) ( T(f) - T(\kappa_1/ \nu_8))}{(T(f) - \nu_3 )(T(f) - \nu_4)}. \nonumber
\end{align}
Namely the operator $T$ represents the time evolution of $q$-$P(D^{(1)}_5)$ (see Eq.~(\ref{eq:qPD15})).
On the other parameters, we have
\begin{align}
& \displaystyle T(\nu _i)= \nu _i \; (i=1,2,\dots ,8) , \;  T(\kappa _1 ) =  \kappa _1  / q , \; T(\kappa _2 ) = q \kappa _2  . \label{eq:D5T}
\end{align}
(ii) Let $G[s] $ and $D[c] $ be the operators in Eqs.~(\ref{eq:Gs}) and (\ref{eq:Dc}).
For any elements $ \zeta  $ generated by $\nu_1, \nu_2, \nu_3, \nu_4, \nu_5 /\kappa _2, \nu_6 /\kappa _2,  \nu_7 /\kappa _1$, $\nu_8 /\kappa _1$, $f$ and $g$, we have
\begin{equation}
 \Xi (\zeta ) =  G \Bigl[ \frac{\kappa _2 }{\nu _5 \nu _6} \Bigr] D \Bigl[ \frac{\kappa _1 }{q \nu _7 \nu _8 } \Bigr](\zeta ).
\end{equation}
\end{thm}
\begin{proof}
Set $s= \pi _2 \pi _1 s_2 s_1 s_0 s_2 $.
It follows from Eq.~(\ref{eq:AffineWeylActionD5}) that
\begin{align}
& s (\nu_1) = \nu_7 , \: s (\nu_2) =\nu_8, \: s (\nu_3) = \kappa_2 /\nu_6 , \: s (\nu_4) = \kappa_2 / \nu_5 , \\
& s (\nu_5) = \nu_3 , \: s (\nu_6) =\nu_4 , \: s (\nu_7 ) = \kappa_2 /\nu_2, \: s (\nu_8 ) = \kappa_2 /\nu_1 , \nonumber \\
& s (\kappa _1 ) = \kappa_2 , \: s  (\kappa _2 ) = \kappa_1 \kappa_2^2 /(\nu_1 \nu_2 \nu_5 \nu_6) = q \nu_3 \nu_4 \nu_7 \nu_8 / \kappa_1  . \nonumber
\end{align}
Then
\begin{align}
& s ^2 (\nu_1) = \kappa_2 /\nu_2 , \: s^2 (\nu_2) =\kappa_2 /\nu_1 , \: s^2 (\nu_3) = q \nu_3 \nu_7 \nu_8 / \kappa_1  , \: s ^2  (\nu_4) =  q \nu_4 \nu_7 \nu_8 / \kappa_1 , \\
& s^2 (\nu_5) = \kappa_2 /\nu_6 , \: s^2 (\nu_6) =\kappa_2 / \nu_5  , \: s^2 (\nu_7 ) = q \nu_3 \nu_4 \nu_7 / \kappa_1 , \: s ^2(\nu_8 ) = q \nu_3 \nu_4 \nu_8 / \kappa_1 , \nonumber \\
& s ^2 (\kappa _1 ) =  q \nu_3 \nu_4 \nu_7 \nu_8 / \kappa_1  , \: s ^2 (\kappa _2 ) = q \kappa_2^3  / (\nu_1 \nu_2 \nu_5 \nu_6 ) . \nonumber
\end{align}
Let $\Xi $ be the operator defined in Eq.~(\ref{eq:D5Xi}).
We have 
\begin{equation}
\Xi  s^2 (\nu _i)= \nu _i, \; (i=1,2,\dots ,8), \; \Xi  s^2 (\kappa _1)= \kappa _1/q , \; \Xi s^2 (\kappa _2)= q \kappa _2 . \label{eq:D5Xis2}
\end{equation}
Therefore, we obtain Eq.~(\ref{eq:D5T}).

It follows from Eq.~(\ref{eq:AffineWeylActionD5}) that $s (f) =1/g $ and 
\begin{equation}
 s (g ) = \frac{\kappa_1  f }{q \nu_3 \nu_4 \nu_7 \nu_8 } \frac{( g -1/\nu_1 )( g -1 /\nu_2 )}{(g -\nu _5 /\kappa _2 )(g  - \nu _6 /\kappa _2)}. \label{eq:D5sg}
\end{equation}
Then 
\begin{equation}
 s^2 (f) = \frac{1}{s(g)} = \frac{q \nu_3 \nu_4 \nu_7 \nu_8 }{ \kappa_1  f} \frac{(g -\nu _5 /\kappa _2 )(g  - \nu _6 /\kappa _2)}{( g -1/\nu_1 )( g -1 /\nu_2 )}.
\end{equation}
Since $\Xi (f) = f \kappa_1 /(q \nu_3 \nu_4 ) $ and $\Xi (g) = g \kappa_2 /(\nu_ 5 \nu_6 ) $, we have
\begin{equation}
\Xi s^2 (f) = \Xi \Bigl( \frac{q \nu_3 \nu_4 \nu_7 \nu_8 }{ f \kappa_1 } \Bigr) \Xi \Bigl( \frac{(g - \nu _5/\kappa _2 )(g  - \nu _6/\kappa _2)}{( g -1/\nu_1 )( g -1 /\nu_2 )} \Bigr) = \frac{\nu_3 \nu_4 }{f} \frac{(g - \nu _5/ \kappa _2 )(g  - \nu _6 /\kappa _2)}{( g -1/\nu_1 )( g -1 /\nu_2 )}.
\end{equation}
Therefore we obtain the first equation of (\ref{eq:D5TfTg}).

It follows from Eq.~(\ref{eq:D5sg}) that
\begin{align}
& s^2 (g)  = s (f) s \Bigl( \frac{ \kappa_1 }{q \nu_3 \nu_4 \nu_7 \nu_8 } \Bigr)  s \Bigl( \frac{( g -1/\nu_1 )( g -1 /\nu_2 )}{(g -/\kappa _2 )(g  - \nu _6 /\kappa _2)} \Bigr) \\
& = \frac{1}{g} \frac{1}{s^2 (\kappa _2)}\frac{( 1/s^2 (f) -1/\nu_7 )(1/s^2 (f) -1 /\nu_8 )}{(1/s^2 (f) -s^2 (1 /\nu _4 ))(1/s^2 (f) -s^2 (1 /\nu _3 ))}. \nonumber
\end{align}
We apply the operation $\Xi $.
It follows from $T=\Xi s^2 $ and Eq.~(\ref{eq:D5Xis2}) that 
\begin{align}
& T (g) = \frac{1}{\Xi ( g )} \frac{1}{q \kappa _2}\frac{( 1/T (f) -1/\Xi (\nu_7 ))(1/T (f) -1 /\Xi (\nu_8 ))}{(1/T (f) - 1 /\nu _4 )(1/T (f) - 1/\nu _3)} \\
& = \frac{\nu _5 \nu _6 }{g q \kappa _2^2}\frac{\nu _3 \nu _4}{\Xi (\nu_7 ) \Xi (\nu_8 )}\frac{( T (f) -\Xi (\nu_8 ))(T (f) -\Xi (\nu_7 ))}{(T (f) - \nu _3 )(T (f) - \nu _4)}\nonumber 
\end{align}
Then the second equation of (\ref{eq:D5TfTg}) follows from $ \Xi (\nu _7)=  \kappa _1 /(q \nu _8)$, $ \Xi (\nu _8)=  \kappa _1 /(q \nu _7)$ and the relation $\kappa _1^2 \kappa _2^2 = q \nu_1 \nu_2 \nu_3 \nu_4 \nu_5 \nu_6 \nu_7 \nu_8 $.
Therefore we obtain (i).

(ii) follows from confirming the relation $\Xi ( \nu_5 /\kappa _2 ) =1 /\nu _6 = G [ \kappa _2 /(\nu _5 \nu _6 )]  ( \nu_5 /\kappa _2 ) $ and so on.
\end{proof}
Theorem \ref{thm:D5TW} asserts that the time evolution is essentially expressed as a composition of Weyl group symmetry and the symmetry originated from a gauge transformation and a dilation.

\section{The $q$-Painlev\'e equation of type $E^{(1)}_6 $} \label{sec:E16}

The $q$-Painlev\'e equation of type $E^{(1)}_{6} $ ($q$-$P(E^{(1)}_{6})$) was given as
\begin{align}
 & \frac{(fg-1)(\overline{f}g-1)}{f\overline{f}}=\frac{(g- 1 / \nu_{1} ) (g- 1 / \nu_{2} ) (g- 1 / \nu_{3} ) (g- 1 / \nu_{4} )}{ (g- \nu_{5}/\kappa_{2} ) (g- \nu_{6}/\kappa_{2} )}, \label{eq:qPE16} \\
 & \frac{(fg-1)(f\underline{g}-1)}{g\underline{g}}=\frac{(f-\nu_{1}) (f-\nu_{2}) (f-\nu_{3}) (f-\nu_{4})}{ (f- \kappa_{1}/\nu_{7} )(f- \kappa_{1}/\nu_{8} )} . \nonumber
\end{align}
The equation $q$-$P(E^{(1)}_{6})$ is realized by the compatibility condition for the Lax pair $L_1$ and $L_2$, where
\begin{align}
    L_{1}=&\frac{ z (g\nu_{1}-1)(g\nu_{2}-1)(g\nu_{3}-1)(g\nu_{4}-1)}{g(fg-1) ( g z - q )}-\frac{( g\kappa_{2}/\nu_{5}-1 )( g\kappa_{2}/\nu_{6}-1 ){\kappa_{1}}^{2}}{q fg \nu_{7}\nu_{8}} \label{eq:E6L1} \\
& +\frac{(\nu_{1}- z/q )(\nu_{2}- z/q )(\nu_{3}- z/q )(\nu_{4}- z/q )}{f- z/q}\Bigl\{\frac{g}{1-g z/q}-T_{z}^{-1}\Bigr\} \nonumber \\
    &+\frac{( \kappa_{1}/\nu_{7}-z )( \kappa_{1}/\nu_{8}-z )}{q(f-z)}\Bigl\{\Bigl(\frac{1}{g}-z\Bigr)-T_{z}\Bigr\} \, , \nonumber  \\
    L_{2}=&\Bigl(1-\frac{f}{z}\Bigr)T+T_{z}-\Bigl(\frac{1}{g}-z\Bigr) \, . \nonumber 
\end{align}
Then the linear $q$-differential equation $L_1 y(z)= 0 $ is written as 
\begin{align}
& \Bigl\{ \frac{ z (g\nu_{1}-1)(g\nu_{2}-1)(g\nu_{3}-1)(g\nu_{4}-1)}{g(fg-1) ( g z - q )}-\frac{( g\kappa_{2}/\nu_{5}-1 )( g\kappa_{2}/\nu_{6}-1 ){\kappa_{1}}^{2}}{q fg \nu_{7}\nu_{8}} \Bigr\} y(z) \label{eq:E6L1yz} \\
& +\frac{(\nu_{1}- z/q )(\nu_{2}- z/q )(\nu_{3}- z/q )(\nu_{4}- z/q )}{f- z/q} \Bigl( \frac{g}{1-g z/q} y(z) -y(z/q) \Bigr) \nonumber \\
& +\frac{( \kappa_{1}/\nu_{7}-z )( \kappa_{1}/\nu_{8}-z )}{q(f-z)}\Bigl( \Bigl(\frac{1}{g}-z\Bigr) y(z)- y(qz) \Bigr) =0 . \nonumber 
\end{align}
Note that Eq.~(\ref{eq:E6L1yz}) was studied in \cite{STT1} from the aspect of the initial-value space of $q$-$P(E^{(1)}_{6} )$, and it was observed that a variant of the $q$-Heun equation appears.

The Weyl group symmetry of the $q$-Painlev\'e equation $q$-$P(E^{(1)}_6 )$ was given by the following action of the operators $s_0, \dots ,s_6$, $\pi _1$ and $\pi _2$ for the parameters $(\kappa_{1}, \kappa_{2}, \nu_{1}, \dots , \nu_{8}) $ and $(f, g) $ in \cite{KNY}.
\begin{align}
s_0 : & \: \nu_7 \leftrightarrow \nu_8 , \quad s_1 : \: \nu_5 \leftrightarrow \nu_6  , \quad s_3 : \: \nu_1 \leftrightarrow \nu_2 , \quad s_4 : \: \nu_2 \leftrightarrow \nu_3 , \quad s_5 : \: \nu_3 \leftrightarrow \nu_4 , \label{eq:AffineWeylActionE6} \\
 s_2 : & \: \nu_1 \rightarrow\frac{\kappa_2}{\nu_6}, \;
 \nu_6 \rightarrow \frac{\kappa_2}{\nu_1}, \;
 \kappa_1 \rightarrow \frac{\kappa_1\kappa_2}{\nu_1\nu_6}, \;
 f \rightarrow f\, \frac{\kappa_2 ( \nu_1 g-1)  }{ - ( \kappa_2 - \nu _1 \nu_6 )fg + \nu _1 \kappa_2 g - \nu _1 \nu_6 } ,\nonumber \\
 s_6 : & \: \nu_1 \rightarrow\frac{\kappa_1}{\nu_7}, \; 
 \nu_7 \rightarrow \frac{\kappa_1}{\nu_1}, \;
 \kappa_2 \rightarrow \frac{\kappa_1\kappa_2}{\nu_1\nu_7}, \;
  g \rightarrow g\, \frac{\nu _7 ( \nu_1 - f )  }{ \kappa_1  -  \nu _7 f + ( \nu _1 \nu_7 - \kappa_1 )fg } , \nonumber \\
  \pi _1 : & \: q \rightarrow 1/q, \: \nu_1 \rightarrow \nu_2 / \kappa_2 , \:  \nu_2 \rightarrow \nu_1 / \kappa_2 ,  \: \nu_3 \rightarrow 1/\nu_6,  \: \nu_4 \rightarrow 1/\nu_5, \: \nu_5 \rightarrow 1/\nu_4,  \nonumber \\
& \: \nu_6 \rightarrow 1/\nu_3, \: \nu_7 \rightarrow 1/\nu_7, \: \nu_8 \rightarrow 1/\nu_8, \: \kappa_1 \rightarrow \nu_1 \nu_2 /(\kappa_1 \kappa_2 ) , \: \kappa_2 \rightarrow 1/\kappa_2 ,\nonumber \\
& \: f \rightarrow \frac{\nu _1 \nu_2 ( 1- fg )}{ \kappa_2 \{ \nu _1 \nu_2 g + f- (\nu _1 + \nu _2 ) fg \} } , \: g \rightarrow \kappa_2 g , \nonumber \\
  \pi _2 : & \: q \rightarrow 1/q, \: \nu_1 \rightarrow 1/\nu_1, \:  \nu_2 \rightarrow 1/\nu_2,  \: \nu_3 \rightarrow 1/\nu_3,  \: \nu_4 \rightarrow 1/\nu_4, \: \nu_5 \rightarrow 1/\nu_8, \nonumber \\
& \: \nu_6 \rightarrow 1/\nu_7, \: \nu_7 \rightarrow 1/\nu_6, \: \nu_8 \rightarrow 1/\nu_5, \: \kappa_1 \rightarrow 1/\kappa_2 , \: \kappa_2 \rightarrow 1/\kappa_1 , \: f \leftrightarrow  g . \nonumber 
\end{align}
Then we can confirm that these operations satisfy the relations of the extended Weyl group $\widetilde{W}(E^{(1)}_{6})$ whose Dynkin diagram is as follows.\\
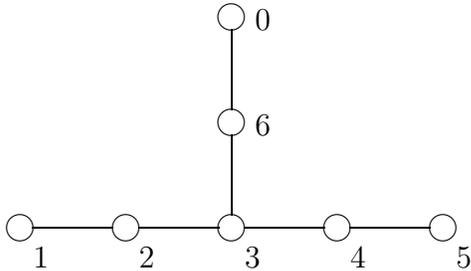
\begin{picture}(0,120)(0,-10)
\put(20,20){\circle{10}}
\qbezier(25,20)(40,20)(55,20)
\put(60,20){\circle{10}}
\qbezier(65,20)(80,20)(95,20)
\put(100,20){\circle{10}}
\qbezier(105,20)(120,20)(135,20)
\put(140,20){\circle{10}}
\qbezier(145,20)(160,20)(175,20)
\put(180,20){\circle{10}}
\qbezier(100,25)(100,40)(100,55)
\put(100,60){\circle{10}}
\qbezier(100,65)(100,80)(100,95)
\put(100,100){\circle{10}}
\put(25,5){$1$}
\put(65,5){$2$}
\put(105,5){$3$}
\put(145,5){$4$}
\put(185,5){$5$}
\put(109,55){$6$}
\put(109,95){$0$}
\end{picture}
We investigate the action of the symmetry in Eq.~(\ref{eq:AffineWeylActionE6}) on the equation $L_1 y(z)= 0 $.
It follows immediately that Eq.~(\ref{eq:E6L1yz}) is invariant under the actions of $s_0 $, $s_1$, $s_3 $, $s_4 $ and $s_5$.
We investigate the action of $s_6$ on Eq.~(\ref{eq:E6L1yz}).
It follows from Eq.~(\ref{eq:AffineWeylActionE6}) that $\nu _1  $ and $\kappa_{1}/\nu_{7} $ are exchanged by the action of $s_6$. 
Thus, we set 
\begin{equation}
 y(z)= \tilde{y}(z) \frac{(q\nu _1 /z;q)_{\infty }}{( q \kappa _1 / (\nu _7 z) ;q)_{\infty }}. \label{eq:E6L1s6}
\end{equation}
Then it is shown as the case of $D^{(1)}_5$ that, if $y(z)$ is a solution of Eq.~(\ref{eq:E6L1yz}), then $\tilde{y}(z)$ satisfies
\begin{align}
& \Bigl\{ \frac{ z (\tilde{g}\tilde{\nu }_{1}-1)(\tilde{g}\nu_{2}-1)(\tilde{g}\nu_{3}-1)(\tilde{g}\nu_{4}-1)}{\tilde{g}(f\tilde{g}-1) ( \tilde{g} z - q )}-\frac{( \tilde{g}\tilde{\kappa }_{2}/\nu_{5}-1 )( \tilde{g}\tilde{\kappa }_{2}/\nu_{6}-1 ){\kappa_{1}}^{2}}{q f\tilde{g} \tilde{\nu }_{7}\nu_{8}} \Bigr\} \tilde{y}(z) \\
& +\frac{(\tilde{\nu }_{1}- z/q )(\nu_{2}- z/q )(\nu_{3}- z/q )(\nu_{4}- z/q )}{f- z/q} \Bigl( \frac{\tilde{g}}{1-\tilde{g} z/q} \tilde{y}(z) -\tilde{y}(z/q) \Bigr) \nonumber \\
& +\frac{( \kappa_{1}/\tilde{\nu }_{7}-z )( \kappa_{1}/\nu_{8}-z )}{q(f-z)}\Bigl( \Bigl(\frac{1}{\tilde{g}}-z\Bigr) \tilde{y}(z)- \tilde{y}(qz) \Bigr) =0 , \nonumber 
\end{align}
where
\begin{align}
& \tilde{g}= g\, \frac{\nu _7 ( \nu_1 - f )  }{ \kappa_1  -  \nu _7 f + ( \nu _1 \nu_7 - \kappa_1 )fg }, \; \tilde{\nu }_1 =\frac{\kappa_1}{\nu_7}, \; \tilde{\nu }_7 =\frac{\kappa_1}{\nu_1}, \tilde{\kappa }_2 = \frac{\kappa_1\kappa_2}{\nu_1\nu_7} .
\end{align}
Therefore, the gauge transformation defined in Eq.~(\ref{eq:E6L1s6}) induces the symmetry $s_6$ of the $q$-Painlev\'e equation of type $E^{(1)}_{6}$.

We investigate the actions of the gauge transformation and dilation in Eq.~(\ref{eq:AffineWeylActionE6}) on the equation $L_1 y(z)= 0 $.
Set $z =u/c$, $y(z)= z^d \tilde{y}(u) = z^d \tilde{y}(c z) $. Then the function $\tilde{y}(u)$ satisfies
\begin{align}
 & \left\{  \frac{  u (g \nu _{1}-1)(g \nu _{2}-1)(g \nu _{3}-1)(g \nu _{4}-1)}{q (g/c) (fg-1)( u g /(c q ) -1 )}- c^2 \frac{ ( g \kappa_{2} /\nu_{5} -1 ) ( g \kappa_{2}/\nu_{6} - 1) {\kappa_{1}}^{2}}{q f g \nu_{7} \nu_{8}} \right\} \tilde{y}(u) \\
& + \frac{ ( c \nu_{1}- u/q )(c \nu_{2}- u/q )(c \nu_{3}- u/q )(c \nu_{4}- u/q ) }{c f- u/q}\left\{\frac{ g/c }{1-u g/(c q) } \tilde{y}(u) - \frac{\tilde{y}(u/q)}{c q^{d}} \right\} \nonumber \\
  & +\frac{( c \kappa_{1}/\nu_{7} -u )( c \kappa_{1}/\nu_{8} -u )}{q(c f-u)}\left\{\left(\frac{c}{g}-u \right) \tilde{y}(u) - c q^d \tilde{y}(q u) \right\} =0 . \nonumber 
\end{align}
If $c q^d =1$, then it is written in the form of Eq.(\ref{eq:E6L1yz}), where the parameters are changed as 
\begin{align}
& (\nu_1 , \nu_2 , \nu_3 , \nu_4 , \nu_5/\kappa_2 ,\nu_6/\kappa_2 , \nu_7/\kappa_1 ,\nu_8/\kappa_1,f,g) \label{eq:E6S} \\
& \mapsto (c  \nu_1 , c \nu_2 , c \nu_3 , c \nu_4 , \nu_5/\kappa_2/c  ,\nu_6/\kappa_2 /c , \nu_7/\kappa_1 /c ,\nu_8/\kappa_1 /c ,cf,g/c) . \nonumber
\end{align}
We denote the transformation of the parameters in Eq.~(\ref{eq:E6S}) by $S_{E_6}[c] $.

The time evolution of the $q$-Painlev\'e equation of type $E^{(1)}_6 $ is expressed as the following theorem.
\begin{thm} (i)
Let $\Xi $ be the transformation of the parameters defined by
\begin{align}
& (\nu_1 , \nu_2 , \nu_3 , \nu_4 ,\nu_5 ,\nu_6 , \nu_7 ,\nu_8 ;\kappa_1 ,\kappa_2 ; f ,g) \label{eq:E6Xi} \\
& \mapsto  \Bigl( \frac{\nu _1 \kappa_2 }{ \nu_ 5 \nu_ 6 \kappa _1^2 } , \frac{\nu _2 \kappa_2 }{ \nu_ 5 \nu_ 6 \kappa _1^2 } , \frac{\nu _3 \kappa_2 }{ \nu_ 5 \nu_ 6 \kappa _1^2 } , \frac{\nu _4 \kappa_2 }{ \nu_ 5 \nu_ 6 \kappa _1^2 } , \frac{q \nu _5 \kappa_1 }{\kappa_2 } ,\frac{q \nu _6 \kappa_1 }{\kappa_2 } , \frac{\kappa_1 }{q \nu_8 } , \frac{\kappa_1 }{q \nu_7 }  ; \nonumber \\
& \qquad \qquad \qquad \qquad \qquad \frac{\kappa_2}{q \nu_5 \nu_6 \nu_7 \nu_8 } ,\frac{\kappa_2  }{q \nu_5 \nu_6 \kappa_1 } ; \frac{f \kappa_2 }{\nu_ 5 \nu_ 6 \kappa _1^2 } , \frac{g \nu_ 5 \nu_ 6 \kappa _1^2 }{\kappa_2 } \Bigr) .  \nonumber
\end{align}
Let $T$ be the transformation of the parameters defined by 
\begin{equation}
T= \Xi \cdot (\pi _1 \pi _2 s_4 s_5 s_3 s_6 s_4 s_3 s_0 s_6 )^2 . \label{eq:E6TXis}
\end{equation}
Then 
\begin{align}
& \frac{(T(f) g-1)(fg-1)}{T(f) f} = \frac{(g -1/\nu_1)(g -1/\nu_2)(g -1/\nu_3)(g -1/\nu_4)}{(g-\nu_5/\kappa_2)(g-\nu_6/\kappa_2)}, \label{eq:E6TfTg} \\
& \frac{(T(f) g- 1 )(T(f) T(g) - 1 )}{g T (g)} = \frac{(T(f) - \nu _1 ) (T(f) - \nu _2 ) (T(f) - \nu _3 ) (T(f) - \nu _4 )}{(T(f)-T(\kappa_1 / \nu_7 ))(T(f)-T(\kappa_1 / \nu_8 ))} . \nonumber
\end{align}
Namely the operator $T$ represents the time evolution of $q$-$P(E^{(1)}_6)$  (see Eq.~(\ref{eq:qPE16})).
On the other parameters, we have
\begin{align}
& \displaystyle T(\nu _i)= \nu _i \; (i=1,2,\dots ,8) , \;  T(\kappa _1 ) =  \kappa _1  / q , \; T(\kappa _2 ) = q \kappa _2  . \label{eq:E6T}
\end{align}
(ii) For any element $\zeta $ generated by $\nu_1, \nu_2, \nu_3, \nu_4, \nu_5 /\kappa _2, \nu_6 /\kappa _2,  \nu_7 /\kappa _1$, $\nu_8 /\kappa _1 $, $f$ and $g$, we have
\begin{equation}
 \Xi (\zeta ) = S_{E_6} \Bigl[ \frac{ \kappa_2 }{ \nu_ 5 \nu_ 6 \kappa _1^2 } \Bigr](\zeta ).
\end{equation}
\end{thm}
\begin{proof}
Set $s= \pi _1 \pi _2 s_4 s_5 s_3 s_6 s_4 s_3 s_0 s_6 $.
It follows from Eq.~(\ref{eq:AffineWeylActionE6}) that
\begin{align}
& s (\nu_1) = \kappa_2 /\nu_4 , \: s (\nu_2) =\kappa_2 /\nu_3 , \: s (\nu_3) = \kappa_2/\nu_2  , \: s (\nu_4) = \kappa_2/\nu_1, \\
& s (\nu_5) = \nu_8 , \: s (\nu_6) =\nu_7 , \: s (\nu_7 ) = \kappa_2 /\nu_5, \: s (\nu_8 ) = \kappa_2 /\nu_6 , \nonumber \\
& s (\kappa _1 ) = \kappa_2 , \: s  (\kappa _2 ) = \kappa_1 \kappa_2^3 /(\nu_1 \nu_2 \nu_3 \nu_4 \nu_5 \nu_6) = q  \nu_7 \nu_8 \kappa_2 / \kappa_1  . \nonumber
\end{align}
Then
\begin{align}
& s ^2 (\nu_1) = q \nu_1 \nu_7 \nu_8 /\kappa_1 , \: s^2 (\nu_2) = q \nu_2 \nu_7 \nu_8 /\kappa_1 , \: s^2 (\nu_3) = q \nu_3 \nu_7 \nu_8 / \kappa_1  , \: s ^2  (\nu_4) =  q \nu_4 \nu_7 \nu_8 / \kappa_1 , \\
& s^2 (\nu_5) = \kappa_2 /\nu_6 , \: s^2 (\nu_6) =\kappa_2 / \nu_5  , \: s^2 (\nu_7 ) = q \nu_7 \kappa_2 / \kappa_1 , \: s ^2(\nu_8 ) =q \nu_8 \kappa_2 / \kappa_1 , \nonumber \\
& s ^2 (\kappa _1 ) =  q \nu_7 \nu_8 \kappa_2/ \kappa_1  , \: s ^2 (\kappa _2 ) = q ^2 \nu_7 \nu_8 \kappa_2^2 /(\nu_5 \nu_6 \kappa_1 ) . \nonumber
\end{align}
Let $\Xi $ be the operator defined in Eq.~(\ref{eq:E6Xi}).
Then we have 
\begin{equation}
\Xi  s^2 (\nu _i)= \nu _i, \; (i=1,2,\dots ,8), \; \Xi  s^2 (\kappa _1)= \kappa _1/q , \; \Xi s^2 (\kappa _2)= q \kappa _2 . \label{eq:E6Xis2}
\end{equation}
Hence we obtain Eq.~(\ref{eq:E6T}).

It follows from Eq.~(\ref{eq:AffineWeylActionE6}) that $s (f)  = \kappa_2 g  $ and 
\begin{equation}
\frac{1}{\kappa_2 s (g)}  = g +\frac{f(g -1/\nu_1)(g -1/\nu_2)(g -1/\nu_3)(g -1/\nu_4)}{(1-fg)(g-\nu_5/\kappa_2)(g-\nu_6/\kappa_2)} .  \label{eq:E6sg}
\end{equation}
Then 
\begin{equation}
\frac{1}{s^2 (f)}  = \frac{1}{s(\kappa_2 ) s(g )} =  \frac{ \kappa_1}{ q  \nu_7 \nu_8 } \Bigl\{ g +\frac{f(g -1/\nu_1)(g -1/\nu_2)(g -1/\nu_3)(g -1/\nu_4)}{(1-fg)(g-\nu_5/\kappa_2)(g-\nu_6/\kappa_2)} \Bigr\} .
\end{equation}
It follows from Eq.~(\ref{eq:E6Xi}) that 
\begin{align}
& \frac{1}{T (f)}  = \frac{1}{\Xi s^2 (f)}  = \Xi \Bigl(  \frac{ \kappa_1}{ q  \nu_7 \nu_8 } \Bigr)  \Xi  \Big\{ g +\frac{f(g -1/\nu_1)(g -1/\nu_2)(g -1/\nu_3)(g -1/\nu_4)}{(1-fg)(g-\nu_5/\kappa_2)(g-\nu_6/\kappa_2)} \Bigr\} \\
& = \frac{\kappa_2 }{\nu_ 5 \nu_ 6 \kappa _1^2 } \frac{\nu_ 5 \nu_ 6 \kappa _1^2 }{\kappa_2 } \Big\{ g +\frac{f(g -1/\nu_1)(g -1/\nu_2)(g -1/\nu_3)(g -1/\nu_4)}{(1-fg)(g-\nu_5/\kappa_2)(g-\nu_6/\kappa_2)} \Bigr\} . \nonumber 
\end{align}
It is equivalent to
\begin{equation}
\frac{(T(f) g-1)(fg-1)}{T(f) f} = \frac{(g -1/\nu_1)(g -1/\nu_2)(g -1/\nu_3)(g -1/\nu_4)}{(g-\nu_5/\kappa_2)(g-\nu_6/\kappa_2)}, 
\end{equation}
and we obtain the first equation of (\ref{eq:E6TfTg}).

It follows from Eq.~(\ref{eq:E6sg}) that
\begin{align}
& \frac{1}{s^2 (g)}  = s (\kappa_2 ) s \Bigl( g +\frac{f(g -1/\nu_1)(g -1/\nu_2)(g -1/\nu_3)(g -1/\nu_4)}{(1-fg)(g-\nu_5/\kappa_2)(g-\nu_6/\kappa_2)} \Bigr) \\
& = s (\kappa_2 g ) + \frac{s(f) \prod_{j=1}^4 (s(\kappa_2  g) - s(\kappa_2 /\nu _j))}{ s(\kappa _2 ) (1-s(f)s(g))(s(\kappa_2 g)-s(\nu_5))(s(\kappa_2 g)-s(\nu_6 ))} \nonumber \\
& = s^2 (f ) + \frac{\kappa_2  g ( s^2 (f )  - s^2 (\nu_4))( s^2 (f )  -s^2 (\nu_3 ))( s^2 (f )  - s^2 (\nu_2))( s^2 (f )  -s^2(\nu_1))}{ (s(\kappa _2 )- \kappa_2  g  s^2 (f ))( s^2 (f ) - \nu_8 )( s^2 (f ) - \nu_7 )} . \nonumber 
\end{align}
We apply the operation $\Xi $.
Then we have
\begin{align}
& \frac{1}{T (g)} = T(f) +  \frac{(T(f) - T(\nu_4))(T(f) -T(\nu_3 ))(T(f) -T(\nu_2))(T(f) -T (\nu_1))}{(\Xi (q  \nu_7 \nu_8 \kappa_2 / \kappa_1 )/\Xi (\kappa_2 g ) - T(f))(T(f)-\Xi (\nu_8))(T(f)-\Xi s(\nu_7))} \\
& = T(f) +  \frac{g(T(f) - \nu _1 ) (T(f) - \nu _2 ) (T(f) - \nu _3 ) (T(f) - \nu _4 )}{(1 - T(f)g)(T(f)-\kappa_1 /(q \nu_7 ))(T(f)-\kappa_1 /(q \nu_8 ))} . \nonumber 
\end{align}
Therefore we obtain the second equation of (\ref{eq:E6TfTg}).

(ii) follows similarly to the case $D^{(1)}_5 $.
\end{proof}
Note that the expression $(\pi _1 \pi _2 s_4 s_5 s_3 s_6 s_4 s_3 s_0 s_6 )^2 $ in Eq.~(\ref{eq:E6TXis}) can be essentially found in Tsuda's paper \cite{Td}.

\section{The $q$-Painlev\'e equation of type $E^{(1)}_7 $} \label{sec:E17}

The $q$-Painlev\'e equation of type $E^{(1)}_{7} $ ($q$-$P(E^{(1)}_{7})$) was given as
\begin{align}
& \frac{(fg- \kappa_{1}/\kappa_{2} ) (\overline{f}g- \kappa_{1}/(q\kappa_{2}) )}{(fg-1)(\overline{f}g-1)}=\frac{(g- \nu_{5}/\kappa_{2} ) (g- \nu_{6}/\kappa_{2} ) (g- \nu_{7}/\kappa_{2} ) (g- \nu_{8}/\kappa_{2} )}{ (g - 1/\nu_{1} )(g - 1/\nu_{2} )(g - 1/\nu_{3} )(g - 1/\nu_{4} )}, \label{eq:qPE17} \\
& \frac{(fg-\kappa_{1}/\kappa_{2} ) (f\underline{g}- q\kappa_{1}/\kappa_{2} )}{(fg-1)(f\underline{g}-1)}=\frac{(f- \kappa_{1}/\nu_{5})(f- \kappa_{1}/\nu_{6})(f- \kappa_{1}/\nu_{7})(f- \kappa_{1}/\nu_{8})}{(f-\nu_{1})(f-\nu_{2})(f-\nu_{3})(f-\nu_{4})}  . \nonumber 
\end{align}
The equation $q$-$P(E^{(1)}_{7})$ is realized by the compatibility condition for the Lax pair $L_1$ and $L_2$, where
\begin{align}
& L_{1}=\frac{q(\kappa_{1}-\kappa_{2})(g\kappa_{2}-\nu_{5})(g\kappa_{2}-\nu_{6})(g\kappa_{2}-\nu_{7})(g\kappa_{2}-\nu_{8})}{g\kappa_{1}{\kappa_{2}}^{2}(fg\kappa_{2}-\kappa_{1})(g\kappa_{2}z-\kappa_{1})} \label{eq:E7L1} \\
& \quad -\frac{q(\kappa_{1}-\kappa_{2})(g\nu_{1}-1)(g\nu_{2}-1)(g\nu_{3}-1)(g\nu_{4}-1)}{g(fg-1)\kappa_{1}\nu_{1}\nu_{2}\nu_{3}\nu_{4}(gz-q)} \nonumber \\
& \quad +\frac{(q\nu_{1}-z)(q\nu_{2}-z)(q\nu_{3}-z)(q\nu_{4}-z)\{(g\kappa_{2}z-\kappa_{1}q)- \kappa_{1} (gz-q)T_{z}^{-1}\}}{\kappa_{1}q\nu_{1}\nu_{2}\nu_{3}\nu_{4}(fq-z)z^{2}(q-gz)} \nonumber \\
& \quad -\frac{q(\kappa_{1}-\nu_{5}z)(\kappa_{1}-\nu_{6}z)(\kappa_{1}-\nu_{7}z)(\kappa_{1}-\nu_{8}z)\{\kappa_{1}(gz-1)-(g\kappa_{2}z-\kappa_{1})T_{z}\}}{{\kappa_{1}}^{4}(f-z)z^{2}(g\kappa_{2}z -\kappa_{1})} \, , \nonumber \\
& L_{2}=(1-zg \kappa_{2}/\kappa_{1})T_{z}-(1-zg)+z(z-f)gT \, .\nonumber 
\end{align}
Then the linear $q$-differential equation $L_1 y(z)= 0 $ is written as 
\begin{align}
& \Bigl\{ \frac{q(\kappa_{1}-\kappa_{2})(g\kappa_{2}-\nu_{5})(g\kappa_{2}-\nu_{6})(g\kappa_{2}-\nu_{7})(g\kappa_{2}-\nu_{8})}{g\kappa_{1}{\kappa_{2}}^{2}(fg\kappa_{2}-\kappa_{1})(g\kappa_{2}z-\kappa_{1})} \label{eq:E7L1yz} \\
& -\frac{q(\kappa_{1}-\kappa_{2})(g\nu_{1}-1)(g\nu_{2}-1)(g\nu_{3}-1)(g\nu_{4}-1)}{g(fg-1)\kappa_{1}\nu_{1}\nu_{2}\nu_{3}\nu_{4}(gz-q)}\Bigr\} y(z)  \nonumber \\
& +\frac{(q\nu_{1}-z)(q\nu_{2}-z)(q\nu_{3}-z)(q\nu_{4}-z)}{q\nu_{1}\nu_{2}\nu_{3}\nu_{4}(fq-z)z^{2}} \Bigl\{ \frac{g\kappa_{2}z-\kappa_{1}q}{\kappa_{1}(q-gz)} y(z) +  y(z/q) \Bigr\} \nonumber \\
&  +\frac{q(\kappa_{1}-\nu_{5}z)(\kappa_{1}-\nu_{6}z)(\kappa_{1}-\nu_{7}z)(\kappa_{1}-\nu_{8}z)}{{\kappa_{1}}^{4}(f-z)z^{2}} \Bigl\{ \frac{\kappa_{1}(1-gz)}{g\kappa_{2}z -\kappa_{1}} y(z) +  y(qz) \Bigr\} =0 . \nonumber 
\end{align}
Note that Eq.~(\ref{eq:E7L1yz}) was studied in \cite{STT1} from the aspect of the initial-value space of $q$-$P(E^{(1)}_{7} )$, and it was observed that a variant of the $q$-Heun equation appears.

The Weyl group symmetry of the $q$-Painlev\'e equation $q$-$P(E^{(1)}_7 )$ was given by the following action of the operators $s_0, \dots ,s_6$, $\pi _1$ and $\pi _2$ for the parameters $(\kappa_{1}, \kappa_{2}, \nu_{1}, \dots , \nu_{8}) $ and $(f, g) $ in \cite{KNY}.
\begin{align}
 s_0 : & \: \kappa_1 \leftrightarrow \kappa_2 ,  \: f \rightarrow 1/g , \: g \rightarrow 1/f ,\label{eq:AffineWeylActionE7} \\
  s_1 : & \: \nu_3 \leftrightarrow \nu_4 , \quad   s_2 : \: \nu_2 \leftrightarrow \nu_3 , \quad  s_3 : \nu_1 \leftrightarrow \nu_2 ,\nonumber \\
  s_5 : & \: \nu_5 \leftrightarrow \nu_6 , \quad   s_6 : \: \nu_6 \leftrightarrow \nu_7 , \quad   s_7 : \: \nu_7 \leftrightarrow \nu_8 \nonumber \\
  s_4 : & \: \nu_1 \rightarrow\frac{\kappa_2}{\nu_5}, \;
  \nu_5 \rightarrow \frac{\kappa_2}{\nu_1}, \;
  \kappa_1 \rightarrow \frac{\kappa_1\kappa_2}{\nu_1\nu_5}, \nonumber \\
& \:    f \rightarrow \frac{- \kappa_2 ( \nu_1 \nu_5 - \kappa_1 )f g - \nu_5 ( \kappa_1 - \kappa_2 )f + \kappa_1 ( \nu_1 \nu_5 - \kappa_2 )}{\nu _5 \{ - ( \nu_1 \nu_5 - \kappa_2 )f g + \nu_1 ( \kappa_1 - \kappa_2 )g + ( \nu_1 \nu_5 - \kappa_1 ) \}} , \nonumber \\
  \pi  : & \: q \rightarrow 1/q, \: \nu_1 \rightarrow 1/\nu_5, \:  \nu_2 \rightarrow 1/\nu_6,  \: \nu_3 \rightarrow 1/\nu_7,  \: \nu_4 \rightarrow 1/\nu_8, \: \nu_5 \rightarrow 1/\nu_1, \: \nu_6 \rightarrow 1/\nu_2, \nonumber \\
& \: \nu_7 \rightarrow 1/\nu_3, \: \nu_8 \rightarrow 1/\nu_4, \: \kappa_1 \rightarrow 1/\kappa_1 , \: \kappa_2 \rightarrow 1/\kappa_2 , \: f \rightarrow f/\kappa_1  , \: g \rightarrow \kappa_2 g , \nonumber 
\end{align}
Then we can confirm that these operations satisfy the relations of the extended Weyl group $\widetilde{W}(E^{(1)}_{7})$ whose Dynkin diagram is as follows.\\
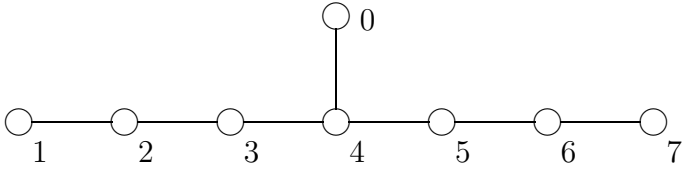
\begin{picture}(0,70)(0,0)
\put(20,20){\circle{10}}
\qbezier(25,20)(40,20)(55,20)
\put(60,20){\circle{10}}
\qbezier(65,20)(80,20)(95,20)
\put(100,20){\circle{10}}
\qbezier(105,20)(120,20)(135,20)
\put(140,20){\circle{10}}
\qbezier(145,20)(160,20)(175,20)
\put(180,20){\circle{10}}
\qbezier(185,20)(200,20)(215,20)
\put(220,20){\circle{10}}
\qbezier(225,20)(240,20)(255,20)
\put(260,20){\circle{10}}
\qbezier(140,25)(140,40)(140,55)
\put(140,60){\circle{10}}
\put(25,5){$1$}
\put(65,5){$2$}
\put(105,5){$3$}
\put(145,5){$4$}
\put(185,5){$5$}
\put(225,5){$6$}
\put(265,5){$7$}
\put(149,55){$0$}
\end{picture}

We investigate the action of the symmetry in Eq.~(\ref{eq:AffineWeylActionE7}) on the equation $L_1 y(z)= 0 $.
It follows immediately that Eq.~(\ref{eq:E7L1yz}) is invariant under the actions of $s_1 $, $s_2$, $s_3 $, $s_5 $, $s_6 $ and $s_7$.
The action of $s_4$ does not exchange $\nu_1$ and $\kappa_{1}/ \nu_{5}$, which are related with zeros of the coefficient of $y(z/q)$ and $y(qz)$, and the action of $s_0 s_4 s_0$ exchanges them.
More precisely, the action is written as
\begin{align}
s_0 s_4 s_0 : & \: \nu_1 \rightarrow\frac{\kappa_1}{\nu_5}, \;
    \nu_5 \rightarrow \frac{\kappa_1}{\nu_1}, \;
    \kappa_2 \rightarrow \frac{\kappa_1\kappa_2}{\nu_1\nu_5},  \\
& \:    g \rightarrow \frac{\nu _5 \{ - ( \nu_1 \nu_5 - \kappa_1 ) + \nu_1 ( \kappa_2 - \kappa_1 )g + ( \nu_1 \nu_5 - \kappa_2 )fg \}}{- \kappa_1 ( \nu_1 \nu_5 - \kappa_2 ) - \nu_5 ( \kappa_2 - \kappa_1 )f + \kappa_2 ( \nu_1 \nu_5 - \kappa_1 )fg } . \nonumber
\end{align}
Set 
\begin{equation}
 y(z)= \tilde{y}(z) \frac{(q\nu _1 /z;q)_{\infty }}{( q \kappa _1 / (\nu _5 z) ;q)_{\infty }}. \label{eq:E7L1s6}
\end{equation}
It follows that, if $y(z)$ is a solution of Eq.~(\ref{eq:E7L1yz}), then $\tilde{y}(z)$ satisfies the equation as Eq.~(\ref{eq:E7L1yz}) whose parameters are changed by the action of $s_0 s_4 s_0 $.

We investigate the action of another gauge transformation in Eq.~(\ref{eq:AffineWeylActionE7}) on the equation $L_1 y(z)= 0 $.
Set $z =u/c$ and $y(z)= \tilde{y}(u)$.
If $y(z) $ is a solution of Eq.~(\ref{eq:E7L1yz}), then the function $\tilde{y}(u)$ satisfies Eq.~(\ref{eq:E7L1yz}) for the variable $u$, where the parameters are changed as 
\begin{align}
& (\nu_1 , \nu_2 , \nu_3 , \nu_4 , \nu_5/\kappa_1 ,\nu_6/\kappa_1 , \nu_7/\kappa_1 ,\nu_8/\kappa_1, \kappa_2/\kappa_1, f,g)  \label{eq:E7S} \\
& \mapsto (c  \nu_1 , c \nu_2 , c \nu_3 , c \nu_4 , (\nu_5/\kappa_1)/c  , (\nu_6/\kappa_1 )/c , (\nu_7/\kappa_1 )/c , (\nu_8/\kappa_1) /c , \kappa_2/\kappa_1 ,cf,g/c)  . \nonumber
\end{align}
We denote the transformation of the parameters in Eq.~(\ref{eq:E7S}) by $S_{E_7}[c] $.

The time evolution of the $q$-Painlev\'e equation of type $E^{(1)}_7 $ is expressed as the following theorem.
\begin{thm} (i)
Let $\Xi $ be the transformation of the parameters defined by
\begin{align}
& (\nu_1 , \nu_2 , \nu_3 , \nu_4 ,\nu_5 ,\nu_6 , \nu_7 ,\nu_8 ;\kappa_1 ,\kappa_2 ; f ,g) \label{eq:E7Xi} \\
& \mapsto  \Bigl( \frac{\nu _1 \kappa_1  }{ q \kappa_2 } ,\frac{\nu _2 \kappa_1  }{ q \kappa_2 } ,\frac{\nu _3 \kappa_1  }{ q \kappa_2 } ,\frac{\nu _4 \kappa_1  }{ q \kappa_2 } ,\frac{\nu _5 \kappa_1  }{ q \kappa_2 } ,\frac{\nu _6 \kappa_1  }{ q \kappa_2 } ,\frac{\nu _7 \kappa_1  }{ q \kappa_2 } ,\frac{\nu _8 \kappa_1  }{ q \kappa_2 }  ; \frac{\kappa_1^3}{q^2 \kappa _2^2 } ,\frac{\kappa_1^2 }{q^2 \kappa _2 } ; \frac{f \kappa_1}{q \kappa _2} , \frac{g q \kappa_2 }{\kappa_1 } \Bigr) .  \nonumber
\end{align}
Let $T$ be the transformation of the parameters defined by 
\begin{equation}
T= \Xi \cdot ( s_4 s_5 s_1 s_4 s_6 s_5 s_1 s_2 s_4 s_7 s_6 s_5 s_1 s_2 s_3 s_4 s_0  )^2 . \label{eq:E7TXis}
\end{equation}
Then 
\begin{align}
& \frac{(T(f) g -\kappa_1/(q \kappa_2 ) )(f g -\kappa_1 /\kappa_2 )}{( T(f) g-1 )(fg -1)} = \frac{ ( g -  \nu _5 / \kappa_2 ) ( g -  \nu _6 / \kappa_2 ) ( g -  \nu _7 / \kappa_2 ) ( g -  \nu _8 / \kappa_2 ) }{(g -1/ \nu _1 ) (g -1/ \nu _2 ) (g -1/ \nu _3 ) (g -1/ \nu _4 ) } , \label{eq:E7TfTg} \\
& \frac{( T(g) T(f) -\kappa_1/(q^2 \kappa_2) )( T(f) g - \kappa _1 /( q \kappa_2 ))}{(T(g)T(f) -1)(T (f) g - 1 )} \nonumber \\
& = \frac{ (T (f) - T(\kappa_1 / \nu_5 )) (T (f) - T(\kappa_1 /\nu_6 )) (T (f) - T( \kappa_1 /\nu_7 )) (T (f) - T(\kappa_1 / \nu_8 )) }{ ( T (f) - \nu_1 ) ( T (f) - \nu_2 ) ( T (f) - \nu_3 ) ( T (f) - \nu_4 )}. \nonumber
\end{align}
Namely the operator $T$ represents the time evolution of $q$-$P(E^{(1)}_7)$ (see Eq.~(\ref{eq:qPE17})).
On the other parameters, we have
\begin{align}
& \displaystyle T(\nu _i)= \nu _i \; (i=1,2,\dots ,8) , \;  T(\kappa _1 ) =  \kappa _1  / q , \; T(\kappa _2 ) = q \kappa _2  . \label{eq:E7T} 
\end{align}
(ii) For any elements $ \zeta $ generated by $\nu_1, \nu_2, \nu_3, \nu_4, \nu_5 /\kappa _1, \nu_6 /\kappa _1,  \nu_7 /\kappa _1$, $\nu_8 /\kappa _1 $, $\kappa_2/\kappa_1 $, $f$ and $g$, we have
\begin{equation}
 \Xi (\zeta ) = S_{E_7} \Bigl[ \frac{\kappa_1  }{ q \kappa_2 }\Bigr](\zeta ).
\end{equation}
\end{thm}
\begin{proof}
Set $s= s_4 s_5 s_1 s_4 s_6 s_5 s_1 s_2 s_4 s_7 s_6 s_5 s_1 s_2 s_3 s_4 s_0  $.
Then it follows from Eq.~(\ref{eq:AffineWeylActionE7}) that $s(\nu_i) = \kappa_2 /\nu_{9-i}$ $(i=1,\dots ,8)$, $s (\kappa _1 ) = \kappa_2 $, $s (\kappa _2 ) = q  \kappa_2^2 / \kappa_1 $, $s (f) = 1/g$ and
\begin{equation}
\frac{( s(g)/g -\kappa_1/(q \kappa_2))(fg -1)}{(s(g)/g -1)(f g -\kappa_1 /\kappa_2 )} = \frac{ \kappa_1/(q \kappa_2 ) (g -1/\nu_1)(g -1/\nu_2)(g -1/\nu_3)( g -1/\nu_4 )}{ (g - \nu_5 /\kappa_2)(g - \nu_6 /\kappa_2)(g - \nu_7 /\kappa_2 )(g - \nu_8 /\kappa_2 )} .\label{eq:E7sg}
\end{equation}
Therefore, $s^2(\nu _j) = q \nu _j \kappa_2 /\kappa_1 $ $(j=1,\dots 8)$, $s^2 (\kappa_1) = q \kappa_2^2 /\kappa_1$ and $s^2 (\kappa_2 ) = q^3 \kappa_2^3 /\kappa_1^2$.

Let $\Xi $ be the operator defined in Eq.~(\ref{eq:E7Xi}).
We have 
\begin{equation}
\Xi  s^2 (\nu _i)= \nu _i, \; (i=1,2,\dots ,8), \; \Xi  s^2 (\kappa _1)= \kappa _1/q , \; \Xi s^2 (\kappa _2)= q \kappa _2 . \label{eq:E7Xis2}
\end{equation}
Then we obtain Eq.~(\ref{eq:E7T}).
Note that $\Xi (\kappa_2 /\kappa_1) = \kappa_2 /\kappa_1$ and $\Xi (\nu _i /\kappa_2 ) = (q \kappa_2 /\kappa_1 ) (\nu _i /\kappa_2 )$ $(i=1,\dots ,8)$. 

It follows from Eq.~(\ref{eq:E7sg}) and $s^2 (f)  = 1/ s(g )$ that
\begin{equation}
\frac{(  q \kappa_2/\kappa_1  - g s^2(f) )(fg -1)}{(1-g s^2 (f))(f g -\kappa_1 /\kappa_2 )} = \frac{ (g -1/\nu_1)(g -1/\nu_2)(g -1/\nu_3)( g -1/\nu_4 )}{ (g - \nu_5 /\kappa_2)(g - \nu_6 /\kappa_2)(g - \nu_7 /\kappa_2 )(g - \nu_8 /\kappa_2 )} .
\end{equation}
We apply the operation $\Xi $.
It follows from Eq.~(\ref{eq:E7Xi}) that 
\begin{equation}
\frac{( 1 - g T (f))(fg -1)}{(\kappa_1/(q \kappa_2 )  - g T (f) )(f g -\kappa_1 /\kappa_2 )} = \frac{(g -1/ \nu _1 ) (g -1/ \nu _2 ) (g -1/ \nu _3 ) (g -1/ \nu _4 ) }{ ( g -  \nu _5 / \kappa_2 ) ( g -  \nu _6 / \kappa_2 ) ( g -  \nu _7 / \kappa_2 ) ( g -  \nu _8 / \kappa_2 ) } .
\end{equation}
Therefore we obtain the first equation of (\ref{eq:E7TfTg}).

We apply the operation $s$ to Eq.~(\ref{eq:E7sg}).
It follows from $s (f) = 1/g$ and $s(g ) = 1/ s^2 (f) $ that
\begin{align}
& \frac{( s^2(g)s^2 (f) -\kappa_1/\kappa_2/q^2 )( g s^2 (f) -1)}{(s^2(g)s^2 (f) -1)(g s^2 (f) -q \kappa_2/\kappa_1 )} \\
& = \frac{ (s^2 (f) - \kappa_2 /\nu_5 ) (s^2 (f) - \kappa_2 /\nu_6 ) (s^2 (f) - \kappa_2 /\nu_7 ) (s^2 (f) - \kappa_2 /\nu_8 )  }{ ( s^2 (f) -s^2 ( \nu_1 )) ( s^2 (f) -s^2 ( \nu_2 )) ( s^2 (f) -s^2 ( \nu_3 )) ( s^2 (f) -s^2 ( \nu_4 )) } . \nonumber 
\end{align}
By applying the operation $\Xi $, we obtain the second equation of (\ref{eq:E7TfTg}). 
\end{proof}
Note that the expression $( s_4 s_5 s_1 s_4 s_6 s_5 s_1 s_2 s_4 s_7 s_6 s_5 s_1 s_2 s_3 s_4 s_0  )^2 $ in Eq.~(\ref{eq:E7TXis}) can be essentially found in \cite{Td}.

\section*{Acknowledgements}
The author thanks Shoko Sasaki and Shun Takagi for discussions.
He was supported by JSPS KAKENHI Grant Number JP18K03378.

\end{document}